\documentclass[12pt,reqno]{amsart}

\usepackage{mathrsfs}

\usepackage{amsmath,amsthm,amssymb,wasysym,graphicx,mathtools, enumerate}
\graphicspath{{./pics/}}
\usepackage[inline]{enumitem}

\usepackage{caption}
\usepackage{subcaption}

\DeclareMathOperator{\real}{\mathrm{Re}}
\DeclareMathOperator{\RE}{\mathrm{Re}}
\DeclareMathOperator{\imag}{\mathrm{Im}}
\numberwithin{equation}{section}
\theoremstyle{plain}

\newtheorem{theorem}{Theorem}[section]

\newtheorem{corollary}[theorem]{Corollary}

\theoremstyle{definition}

\allowdisplaybreaks

\makeatletter
\@namedef{subjclassname@2020}{\textup{2020} Mathematics Subject Classification}
\makeatother
\begin{document}
	\title[Radius of Starlikeness]{Radius of starlikeness for some classes containing non-univalent  functions}

\author[S. Yadav]{Shalu Yadav}
\address{Department of Mathematics, National Institute of Technology,
Tiruchirappalli--620 015, India}
\email{shaluyadav903@gmail.com }

\author[K. Sharma]{Kanika Sharma}
\address{Department of Mathematics, Atma Ram Sanatan Dharma College,
University of Delhi, Delhi--110 021, India}
\email{ksharma@arsd.du.ac.in; kanika.divika@gmail.com}

\author{V. Ravichandran}
\address{Department of Mathematics, National Institute of Technology,
Tiruchirappalli--620 015, India} \email{ravic@nitt.edu; vravi68@gmail.com}

\begin{abstract}A starlike univalent function  $f$ is characterized by the function $zf'(z)/f(z)$; several subclasses of these functions were studied in the past by restricting the function $zf'(z)/f(z)$ to take values  in a region $\Omega$ on the right-half plane, or, equivalently, by requiring the function $zf'(z)/f(z)$ to be subordinate to the corresponding mapping of the unit disk $\mathbb{D}$ to the region $\Omega$. 
	The mappings $w_1(z):=z+\sqrt{1+z^2}, w_2(z):=\sqrt{1+z}$ and $w_3(z):=e^z$ maps the unit disk $\mathbb{D}$ to various regions in the right half plane. For normalized analytic functions $f$ satisfying the conditions that $f(z)/g(z), g(z)/zp(z)$ and $p(z)$ are subordinate to the functions $w_i, i=1,2,3$ in various ways for some analytic functions $g(z)$ and $p(z)$, we determine the sharp radius for them to belong to various subclasses of starlike functions.
\end{abstract}



\thanks{The research of  Shalu Yadav  is supported by an institute fellowship from NIT Trichy}

	\maketitle
\section{Introduction}Though complex numbers is not ordered field, the inequalities in the real line
can be extended to complex plane in a natural way using the concept of subordination. Let $f, g$ be two analytic functions defined on the open unit disk $ \mathbb{D}:=\{z\in\mathbb{C}:|z|<1\}$.  The  function $f$ is $\textit{subordinate}$ to the function $g$, written $f\prec g$, if
$f=g\circ w$    for some analytic function $w:\mathbb{D}\to \mathbb{D}$ with $w(0)=0$ (and such a function $w$ is known as a Schwartz function). A univalent  function $f:\mathbb{D}\to\mathbb{C}$ is always locally univalent or, in other words, it has non-vanishing derivative.  Therefore, the study of univalent functions can be restricted to the functions  normalized by $f(0)=f'(0)-1=0$. We let $\mathcal{A}$ denote the class of all analytic functions $f:\mathbb{D}\to\mathbb{C}$ normalized by  the conditions $f(0)=0$ and $f'(0)=1$. Since $f'(0)\neq 0$ for functions $f\in \mathcal{A}$, the functions in the class $\mathcal{A}$ are univalent in some disk centred at the origin.  The largest disk centered at the origin in which $f$ is univalent  is called as the radius of univalence of the function $f$.
Consider a subset $\mathcal{M} \subset  \mathcal{A}$ and a property $P$ (such as univalence, starlikeness, convexity) that the image of the functions in $\mathcal{M}$ may or may not have. It often happens that the image of $\mathbb{D}_r=\{z\in \mathbb{D}:|z|<r\}$ for some $r\leq 1$ has the property $P$; the largest $\rho_f$ such that the functions has the property $P$ in each disk $\mathbb{D}_r$ for $0<r\leq \rho_f$ is the radius of the property $P$ of the function $f$. The number $\rho:=\inf \{\rho_f:f\in \mathcal{M}\}$ is the radius of the property $P$ of the class $\mathcal{M}$; if $\mathcal{G}$ is the class of all $f\in \mathcal{A}$ characterized by the property $P$, then $\rho$ is called the $\mathcal{G}$-radius of the class $\mathcal{M}$. The $\mathcal{G}$-radius of the class $\mathcal{M}$ is denoted   by $R_{\mathcal{G}}(\mathcal{M})$ or simply by $R_{\mathcal{G}}$ if the class $\mathcal{M}$ is clear from the context.

The class $\mathcal{M}$ we are interested in is  characterized by the ratio between functions $f$ and $g$ belonging to  $\mathcal{A}$; several authors \cite{nav,ali2, gan, mac4, mac3, mac2, vravicmft, jain} have studied such classes. 
The classes which we are considering are as follows.
\begin{align*}
\mathcal{T}_1&=\{f\in\mathcal{A}:\frac{f}{g}\prec e^z, \frac{g}{zp}\prec e^z\quad \text{for some}\quad g\in\mathcal{A}\text\quad{and}\quad p\prec\sqrt{1+z}\}\\
\mathcal{T}_2&= \{f\in\mathcal{A}:\frac{f}{g}\prec \sqrt{1+z}, \frac{g}{zp}\prec \sqrt{1+z}\quad \text{for some}\quad g\in\mathcal{A}\quad\text{and}\quad p\prec e^z\}\\
\mathcal{T}_3&=\{f\in\mathcal{A}:\frac{f}{g}\prec z+\sqrt{1+z^2}, \frac{g}{zp}\prec z+\sqrt{1+z^2} \quad\text{for some}\quad g\in\mathcal{A}\quad
\text{and}\quad\\&\quad\quad p\prec z+\sqrt{1+z^2}\}.
\end{align*}These classes are motivated by a recent work of Ali, Sharma and Ravichandran \cite{Kanika} wherein similar classes were investigated. 
 These  classes  contain non-univalent functions  and this makes the study of such functions interesting.  We compute $\mathcal{G}$-radius when $\mathcal{G}$ is one of  the following  subclasses of starlike functions studied recently in the literature.  A starlike univalent function  $f$ is characterized by the condition  $\RE(zf'(z)/f(z))>0$. If we define the class $\mathcal{P}$ as the class of all functions $p(z)=1+c_1z+\cdots$ defined on $\mathbb{D}$ satisfying $\RE p(z)>0$, it follows that a function $f$ is starlike if and only if $zf'(z)/f(z)\in \mathcal{P}$. Several subclasses of starlike functions were studied in the past by restricting the function $zf'(z)/f(z)$ to take values  in a region $\Omega$ on the right-half plane, or, equivalently, by requiring the function $zf'(z)/f(z)$ to be subordinate to the corresponding mapping $\varphi:\mathbb{D}\to\Omega$:  $zf'(z)/f(z)\prec\varphi(z)$. For $\varphi(z)=(1+(1-2\alpha)z)/(1-z)$, $0\leq \alpha<1$, this class is the class $\mathcal{S}^*({\alpha})$ of all starlike functions of order $\alpha$.  Several other subclasses of starlike functions are defined by replacing the superordinate functions $\varphi$ by  functions having nice geometry. For the functions $\varphi$ defined by $\varphi(z):=(1+Az)/(1+Bz)$, with $-1\leq B < A \leq 1$, $e^z, 1+(4/3)z+(2/3)z^2, 1+\sin z, z+\sqrt{1+z^2}, 1+(zk+z^2)/(k^2-kz)$ where $k=\sqrt{2}+1$, $1+(2(\log((1+\sqrt{z})/(1-\sqrt{z}))^2\pi^2), 2/(1+e^{-z})$ and $1+z-z^3/3$, we denote the class of all functions $f\in\mathcal{A}$ with $zf^{\prime}(z)/f(z)\prec \varphi(z)$ respectively by $\mathcal{S}^*[A, B], \mathcal{S}^*_{exp}, \mathcal{S}^*_c,\mathcal{S}^*_{sin}, \mathcal{S}^*_{\leftmoon} \mathcal{S}^*_{R}, \mathcal{S}_p^*, \mathcal{S}_{SG}^*$ and $\mathcal{S}_{N_e}^*$.

 \section{The radius of univalence}
 \subsection{Class $\mathcal{T}_1$} This class consists of the analytic functions $f$ such that $f(z)/g(z)$ is subordinate to $e^z$ and $p$ is subordinate to $\sqrt{1+z}$ for some analytic function $g$. This class is non-empty as the function $f_1:\mathbb{D}\rightarrow\mathbb{C}$, defined by $$f_1(z)= z e^{2z}\sqrt{1+z}$$
  belongs to the class $\mathcal{T}_1$. The function $f_1$ satisfies the subordination condition for this class along with the functions $g_1, p_1:\mathbb{D}\rightarrow \mathbb{C}$ by
 \[g_1(z)=z e^z \sqrt{1+z}\quad\text{and}\quad p_1(z)=\sqrt{1+z}.\]
 The function $f_1$ plays the role of extremal for this class. Since

 \[f_1^{\prime}(z)=\frac{e^{2z}(4z^2+7z+2)}{2\sqrt{1+z}},\] it is clear that $f^{\prime}(-7+\sqrt{17}/8)=0$, and so the radius of univalence is at most $(-7+\sqrt{17})/8$. As the function $f_1$ is not univalent in $\mathbb{D}$, the class $\mathcal{T}_1$ contains non-univalent functions. Here also we find that radius of univalence and the radius of starlikeness for $\mathcal{T}_1$ are same and is $(-7+\sqrt{17})/8\approx -0.359612$.

 To do this, we need to first find a disk into which the disk $\mathbb{D}_r$ is mapped by the function $f\in\mathcal{T}_1$. Let $f\in \mathcal{T}_1$ and define the functions $p_1, p_2$ by $p_1(z)=f(z)/g(z)$ and $p_2(z)=g(z)/zp(z)$. Then $f(z)=zp(z)p_1(z)p_2(z)$ and
 \begin{align}
 \left|\frac{zf^{\prime}(z)}{f(z)}-1\right|\leq \left|\frac{zp^{\prime}(z)}{p(z)}\right|+\left|\frac{zp_1^{\prime}(z)}{p_1(z)}\right|+\left|\frac{zp_2^{\prime}(z)}{p_2(z)}\right|.
 \end{align}
 Using the bounds for $p,p_1, p_2$, we obtain
 \begin{align}\label{disk1}
  \left|\frac{zf^{\prime}(z)}{f(z)}-1\right|\leq\begin{cases}
  \frac{r(5-4r)}{2(1-r)}; \quad r\leq \sqrt{2}-1\\
  \frac{3+r+7r^2+3r^4}{2(1-r^2)}; r\geq\sqrt{2}-1.
  \end{cases}
 \end{align}
  for each function $f\in\mathcal{T}_1$. Clearly, for $|z|=r\leq (-7+\sqrt{17})/8$, we have
  \[
  \left|\frac{zf^{\prime}(z)}{f(z)}-1\right|\leq\frac{r(5-4r)}{2(1-r)}\leq 1.
  \]
  This shows that the radius of starlikeness is at least  $(-7+\sqrt{17})/8$. Since the radius of univalence is atmost $ (-7+\sqrt{17})/8$, it follows that the radius univalence and radius of starlikeness are both equal to $ (-7+\sqrt{17})/8$.
 \subsection{Class $\mathcal{T}_2$} Functions of this class are analytic functions $f$ satisfying the subordinations
 \[\frac{f(z)}{g(z)}\prec \sqrt{1+z},\quad \frac{g(z)}{zp(z)}\prec\sqrt{1+z}\quad\text{and}\quad p(z)\prec e^z,\] where the functions $g\in\mathcal{A}$ and $p\in\mathcal{P}$.

 This class is non-empty as the function $f_2:\mathbb{D}\rightarrow\mathbb{C}$ defined by
 \[f_2(z)=z(1+z)e^z\] belongs to the class $\mathcal{T}_2$. The function $f_2$ satisfies the required subordinations when we define the functions $g_2, p_2:\mathbb{D}\rightarrow\mathbb{C}$ by
 \[g_2(z)=z e^z \sqrt{1+z}\quad\text{and}\quad p_2(z)=e^z.\]
 This function $f_2$ plays the role of extremal for this class. Since
 \[f_2^{\prime}(z)=e^z(1+3z+z^2),\] it is clear that $f_2^{\prime}(-3+\sqrt{5}/2)=0$, and so the radius of univalence is atmost $(-3+\sqrt{5})/2$. We shall show that the radius of univalence and the radius of starlikeness for $\mathcal{T}_2$ are equal and commom value of the radius is precisely $(-3+\sqrt{5})/2\approx -0.381966$.

  To do this, we need to first find a disk into which the disk $\mathbb{D}_r$ is mapped by the function $f\in\mathcal{T}_2$. Let $f\in \mathcal{T}_2$ and define the functions $p_1, p_2$ by $p_1(z)=f(z)/g(z)$ and $p_2(z)=g(z)/zp(z)$. Then $f(z)=zp(z)p_1(z)p_2(z)$ and
 \begin{align}
 \left|\frac{zf^{\prime}(z)}{f(z)}-1\right|\leq \left|\frac{zp^{\prime}(z)}{p(z)}\right|+\left|\frac{zp_1^{\prime}(z)}{p_1(z)}\right|+\left|\frac{zp_2^{\prime}(z)}{p_2(z)}\right|.
 \end{align}
 Using the bounds for $p,p_1, p_2$, we obtain
 \begin{align}\label{disk2}
 \left|\frac{zf^{\prime}(z)}{f(z)}-1\right|\leq\begin{cases}
 \frac{r(-2+r)}{(-1+r)}; \quad r\leq \sqrt{2}-1\\
 \frac{1+4r+6r^2+r^4}{4(1-r^2)}; r\geq\sqrt{2}-1.
 \end{cases}
 \end{align}
 for each function $f\in\mathcal{T}_2$. Clearly, for $|z|=r\leq (-3+\sqrt{5})/2$, we have

\[
\left|\frac{zf^{\prime}(z)}{f(z)}-1\right|\leq\frac{r(-2+r)}{(-1+ r)}\leq 1.
\]
This shows that the radius of starlikeness is at least  $(-3+\sqrt{5})/2$. Since the radius of univalence is atmost $ (-3+\sqrt{5})/2$, it follows that the radius univalence and radius of starlikeness are both equal to $(-3+\sqrt{5})/2$.

\subsection{Class $\mathcal{T}_3$}
Recall that a function $f\in\mathcal{A}$ belongs the class  $\mathcal{T}_3$ if there are functions $g\in\mathcal{A}$ and $p\in \mathcal{P}$ satisfying the subordinations
\[ \frac{f(z)}{g(z)}\prec z+\sqrt{1+z^2}, \quad \frac{g(z)}{zxp(z)}\prec z+\sqrt{1+z^2}\quad \text{and }\quad  p(z)\prec z+\sqrt{1+z^2}.\]
This class is non-empty as the function  $f_3:\mathbb{D}\rightarrow\mathbb{C}$ defined by
\[f_3(z)=z(z+\sqrt{1+z^2})^3\] belongs to the class $\mathcal{T}_3$. Indeed, the function $f_3$ satisfies the required subordinations when we define the functions   $  g_3, p_3:\mathbb{D}\rightarrow\mathbb{C}$ by
\[  g_3(z)=z(z+\sqrt{1+z^2})^2\quad\text{and}\quad p_3(z)=z+\sqrt{1+z^2}.\]
The function $f_3$ plays the role of extremal for this class. Since
\[f'_3(z)=\frac{\left(z+\sqrt{1+z^2}\right)^3 \left(3 z+\sqrt{1+z^2}\right)}{\sqrt{1+z^2}},\]
it is clear that
$f_3'(-1/\sqrt{8})=0$, and so the radius of univalence is at most $1/\sqrt{8}$. Also, since the function $f_3$ is not univalent in $\mathbb{D}$, the class $\mathcal{T}_3$ contains non-univalent functions. Indeed, we shall show that the radius of univalence and the radius of starlikeness for $\mathcal{T}_3$ are equal and the common value of the radius  is  precisely $1/\sqrt{8}\approx 0.353553$.

To do this, we need to first find a disk into which the disk $\mathbb{D}_r$ is mapped by the function $f\in\mathcal{T}_3$. Let $f\in\mathcal{T}_3$ and define the functions $p_1, p_2$ by $p_1(z)=f(z)/g(z)$ and $p_2(z)=g(z)/zp(z)$. Then $f(z)=zp(z)p_1(z)p_2(z)$ and
\begin{align}
\left|\frac{zf^{\prime}(z)}{f(z)}-1\right|\leq \left|\frac{zp^{\prime}(z)}{p(z)}\right|+\left|\frac{zp_1^{\prime}(z)}{p_1(z)}\right|+\left|\frac{zp_2^{\prime}(z)}{p_2(z)}\right|.
\end{align}
For $p\in \mathcal{P} $ with $p(z)\prec z+\sqrt{1+z^2}$,   Afis\ and  Noor  \cite{asif}  have shown that
\[ |p(z)|\leq r+\sqrt{1+r^2}, \quad
\left|\frac{zp^{\prime}(z)}{p(z)}\leq \frac{r}{\sqrt{1+r^2}}\right| \quad (|z|\leq r).\]
Using these inequalities for $p, p_1, p_2\prec z+\sqrt{1+z^2} $, we see that
\begin{align}\label{disk}
\left|\frac{zf^{\prime}(z)}{f(z)}-1\right|\leq \frac{3r}{\sqrt{1+r^2}}, \quad |z|\leq r,
\end{align}
for each function $f\in\mathcal{T}_3$. Clearly, for $|z|=r\leq 1/\sqrt{8}$, we have
\[
\left|\frac{zf^{\prime}(z)}{f(z)}-1\right|\leq \frac{3r}{\sqrt{1+r^2}} \leq 1.
\]
This shows that the radius of starlikeness is at least $1/\sqrt{8}$. Since the radius of univalence is at most $1/\sqrt{8}$, it follows that the radius of univalence and the radius of starlikeness are both equal to $1/\sqrt{8}$.
\section{Radius of starlikeness}
Our  first theorem gives the sharp radius of starlikeness of order $\alpha$ of the  classes $\mathcal{T}_1, \mathcal{T}_2$ and $\mathcal{T}_3$. We shall  show that this  radius is the same for   the   subclass $\mathcal{S}^*_{\alpha}$ consisting of all functions $f  \in \mathcal{S}^*(\alpha)$   satisfying $|zf^{\prime}(z)/f(z)-1|<1-\alpha$.

\begin{theorem}\label{thstar} 
	The following results hold for the classes $\mathcal{S}^*(\alpha)$ and $\mathcal{S}^*_{\alpha}$.
\begin{enumerate}
\item[(i)]$R_{\mathcal{S}^*(\alpha)}(\mathcal{T}_1)=R_{\mathcal{S}^*_\alpha}(\mathcal{T}_1)  =(7-2\alpha-\sqrt{17+4\alpha+4\alpha^2})/8 $
\item[(ii)] $R_{\mathcal{S}^*(\alpha)}(\mathcal{T}_2)=R_{\mathcal{S}^*_\alpha}(\mathcal{T}_2)  =((3-\alpha-\sqrt{5-2\alpha+\alpha^2}))/2 $
\item[(iii)]  $R_{\mathcal{S}^*(\alpha)}(\mathcal{T}_3)=R_{\mathcal{S}^*_\alpha}(\mathcal{T}_3) =(1-\alpha)/(\sqrt{8+2\alpha-\alpha^2})$
\end{enumerate}
\end{theorem}

\begin{proof}$(i)$ The function defined by $m(r)=(4r^2-7r+2)/2(1-r), 0\leq r<1$ is a decreasing function. Let $\rho= R_{\mathcal{S}^*_\alpha}(\mathcal{T}_1)$ is the root of the equation $m(r)=\alpha$. From $\eqref{disk1}$, it follows that
	\[
	\real\frac{zf^{\prime}(z)}{f(z)}\geq \frac{4r^2-7r+2}{2(1-r)}=m(r)\geq m(\rho)=\alpha.
	\]
	This shows that $R_{\mathcal{S}^*_\alpha}(\mathcal{T}_1)$ is atleast $\rho$. At $z=-R_{\mathcal{S}^*_\alpha}(\mathcal{T}_1)=-\rho$, the function $f_1$ satisfies
	\[	\frac{zf_1^{\prime}(z)}{f_1(z)}= \frac{4\rho^2-7\rho+2}{2(1-\rho)}=\alpha.\]
	Thus the result is sharp.
	
	Also,
	\[\left|\frac{zf_1^{\prime}(z)}{f_1(z)}-1\right|= 1-\alpha.\]
	This proves that the radii $R_{\mathcal{S}^*(\alpha)}(\mathcal{T}_1)$ and $R_{\mathcal{S}^*_\alpha}(\mathcal{T}_1)$ are same.
	
	$(ii)$ The function defined by $n(r)=(r^2-3r+1)/(1-r), 0\leq r<1$ is a decreasing function. Let $\rho=R_{\mathcal{S}^*_\alpha}(\mathcal{T}_2)$ is the root of the equation $n(r)=\alpha$. From $\eqref{disk2}$, it follows that
	\[\real\frac{zf^{\prime}(z)}{f(z)}\geq \frac{1-3r+r^2}{1-r}=n(r)\geq n(\rho)=\alpha.\]
	This shows that $R_{\mathcal{S}^*_\alpha}(\mathcal{T}_2)$ is atleast $\rho$. At $z=-R_{\mathcal{S}^*_\alpha}(\mathcal{T}_2)=-\rho$, the function $f_2$ satisfies
	\[\frac{zf_2^{\prime}(z)}{f_2(z)}=\frac{1-3\rho+\rho^2}{1-\rho}=\alpha.\]
	Also, for the function $f_2$ we have
	\[\left|\frac{zf_2^{\prime}(z)}{f_2(z)}-1\right|= 1-\alpha.\]
	Hence, the radii $R_{\mathcal{S}^*(\alpha)}$ and $R_{\mathcal{S}^*_\alpha}$ are same for the class $\mathcal{T}_2$.
	
	$(iii)$ The function $s(r)=1-3r/\sqrt{1+r^2}, 0\leq r<1$ is a decreasing function. Let  $\rho= R_{\mathcal{S}^*_\alpha}(\mathcal{T}_3)$ is the root of the equation $s(r)=\alpha$. From $\eqref{disk}$, it follows
	
	\[\real\left(\frac{zf^{\prime}(z)}{f(z)}\right)\geq 1-\frac{3r}{\sqrt{1+r^2}}=s(r)\geq s (\rho)=\alpha.\]
	This shows that  $R_{\mathcal{S}^*_\alpha}(\mathcal{T}_3)$ is atleast $\rho$. At $z=-R_{\mathcal{S}^*_\alpha}(\mathcal{T}_3)=-\rho$.
For the function $f_3(z)$, we have
\[\frac{zf^{\prime}_3(z)}{f_3(z)}=1-\frac{3\rho}{\sqrt{1+\rho^2}}=\alpha.\]

 Now, considering
 \[\left|\frac{zf^{\prime}(z)}{f(z)}-1\right|\leq\frac{3r}{\sqrt{1+r^2}}\leq 1-\alpha.\]
The function $f_3(z)$ provides sharpness, as
\[\left|\frac{zf^{\prime}_3(z)}{f_3(z)}-1\right|=\frac{3z}{\sqrt{1+z^2}}=1-\alpha.\]
The radii $R_{\mathcal{S}^*(\alpha)}$ and $R_{\mathcal{S}^*_\alpha}$ are same for the class $\mathcal{T}_3$.
\end{proof}
The function $\varphi_{PAR}:\mathbb{D}\rightarrow\mathbb{C}$ given by
\[\varphi_{PAR}:= 1+\frac{2}{\pi^2}\left(\log\frac{1+\sqrt{z}}{1-\sqrt{z}}\right)^2,\quad \imag\sqrt{z}\geq 0\]
maps $\mathbb{D}$ on the parabolic region given by $\varphi_{PAR}(\mathbb{D})=\{w=u+iv:v^2<2u-1=\{w:\real w> |w-1|\}$. The class $\mathcal{S}_p^*:=\mathcal{S}^*(\varphi_{PAR})=\{f\in\mathcal{A}:zf^{\prime}(z)/f(z)\prec\varphi_{PAR}(z)\}$ was introduced by R{\o}nning \cite{ron}, and is known as the class of parabolic starlike functions. The class $\mathcal{S}^*_p$ consists of the functions $f\in\mathcal{A}$ satisfying
\[\real\left(\frac{zf^{\prime}(z)}{f(z)}\right)>\left|\frac{zf^{\prime}(z)}{f(z)}-1\right|, z\in\mathbb{D}.\]
Evidently, every parabolic starlike function is also starlike of order $1/2$.

 Shanmugam and Ravichandran \cite{vravicmft} had proved that for $1/2<a<3/2$, then
\begin{align}\label{parabola}
\{w:|w-a|<a-1/2\}\subset \{w:\real w>|w-1|\}.
\end{align}
The following result gives the radius of parabolic starlikeness for the classes $\mathcal{T}_1, \mathcal{T}_2$ and $\mathcal{T}_3$.
\begin{corollary}\label{thpara}
	The following results holds for the class $\mathcal{S}_p^*$.
	\begin{enumerate}
		\item[(i)] $R_{\mathcal{S}_p^*}(\mathcal{T}_1)= (3-\sqrt{5})/4 \approx 0.190983$
		\item[(ii)] $R_{\mathcal{S}_p^*}(\mathcal{T}_2)= (5-\sqrt{17})/4\approx 0.219224$
		\item[(iii)] $R_{\mathcal{S}_p^*}(\mathcal{T}_3)=1/\sqrt{35}\approx 0.169031$
	\end{enumerate}
\end{corollary}
\begin{proof}
In equation $\eqref{parabola}$, putting $a=1$ gives $\mathcal{S}^*_{1/2}\subset \mathcal{S}^*_p$. Every parabolic starlike function is also starlike function of order $1/2$, whence the inclusion $\mathcal{S}^*_{1/2}\subset \mathcal{S}^*_p\subset \mathcal{S}^*(1/2)$.
Therefore, for any class $\mathcal{F}$, readily $R_{\mathcal{S}^*_{1/2}}(\mathcal{F})\leq R_{\mathcal{S}^*_{p}}(\mathcal{F})\leq R_{\mathcal{S}^*(1/2)}(\mathcal{F})$.
			
			When $\mathcal{F}=\mathcal{T}_i, i=1,2,3$, Theorem $\ref{thstar}$ gives $R_{\mathcal{S}^*(\alpha)}(\mathcal{T}_i)= R_{\mathcal{S}^*_{\alpha}}(\mathcal{T}_i)$. This shows that  $R_{\mathcal{S}^*_{1/2}}(\mathcal{T}_i)= R_{\mathcal{S}^*_{p}}(\mathcal{T}_i)= R_{\mathcal{S}^*(1/2)}(\mathcal{T}_i)$. So for $\alpha=1/2$, from Theorem $\ref{thstar}$, it follows that $R_{\mathcal{S}_p^*}(\mathcal{T}_1)= (3-\sqrt{5})/4, R_{\mathcal{S}_p^*}(\mathcal{T}_2)= (5-\sqrt{17})/4 $ and $ R_{\mathcal{S}_p^*}(\mathcal{T}_3)=1/\sqrt{35}$.
			
$(i)$ For the function $f_1(z)=ze^{2z}\sqrt{1+z}$, at $z=-R_{\mathcal{S}_p^*}(\mathcal{T}_1)=-\rho$, we have
	\[\real\frac{zf^{\prime}_1(z)}{f_1(z)}=\frac{4\rho^2-7\rho+2}{2(1-\rho)}=\frac{1}{2}=\left|\frac{zf^{\prime}_1(z)}{f_1(z)}-1\right|.\]
	Thus, $R_{\mathcal{S}_p^*}(\mathcal{T}_1)\leq (3-\sqrt{5})/4$.

	$(ii)$ For the function $f_2(z)=z(1+z)e^z$, at $z=-R_{\mathcal{S}_p^*}(\mathcal{T}_2)=-\rho$, we obtain
	\[\real\frac{zf^{\prime}_2(z)}{f_2(z)}=\frac{\rho^2+3\rho+1}{(1+\rho)}=\frac{1}{2}=\left|\frac{zf^{\prime}_2(z)}{f_2(z)}-1\right|.\]
	Thus $R_{\mathcal{S}_p^*}(\mathcal{T}_2)\leq (5-\sqrt{17})/4$.

	$(iii)$ For the function $f_3(z)=z(z+\sqrt{1+z^2})^3$, at $z=-R_{\mathcal{S}_p^*}(\mathcal{T}_3)=-\rho$, we have
	\[\real\frac{zf^{\prime}_3(z)}{f_3(z)}=1-\frac{3\rho}{\sqrt{1+\rho^2}}=\frac{1}{2}=\left|\frac{zf^{\prime}_3(z)}{f_3(z)}-1\right|.\]
	This proves that $R_{\mathcal{S}_p^*}(\mathcal{T}_3)\leq1/\sqrt{35}$.
\end{proof}
In 2015, Mendiratta $\textit{et al.}$ \cite{men1} introduced the class of starlike functions associated with the exponential function as $\mathcal{S}^*_e=\mathcal{S}^*(e^z)$ and it satisfies the condition $|\log zf^{\prime}(z)/f(z)<1|$. They had also proved that, for $e^{-1}\leq a \leq (e+e^{-1}/2)$,
\begin{align}\label{exp}
\{w\in\mathbb{C}:|w-a|<a-e^{-1}\}\subseteq \{w\in\mathbb{C}:|\log w|<1\}.
\end{align}
\begin{corollary}\label{thexp}
	The following results hold for the class $\mathcal{S}_e^*$.
	\begin{enumerate}
		\item[(i)] $R_{\mathcal{S}_e^*}(\mathcal{T}_1)= (-2+7e-\sqrt{4+4e+17e^2})/8e\approx0.237983$
		\item[(ii)] $R_{\mathcal{S}_e^*}(\mathcal{T}_2)=(-1+3e-\sqrt{1-2e+5e^2})/2e\approx 0.267302$
		\item[(iii)] $R_{\mathcal{S}_e^*}(\mathcal{T}_3)=(1-2e+e^2)/(-1+2e+8e^2)\approx 0.215546$
	\end{enumerate}
\end{corollary}
\begin{proof}
Mendiratta \textit{et. al} provided the inclusion $\eqref{exp}$, which gives $\mathcal{S}^*_{1/e}\subset \mathcal{S}^*_e$. It was also shown in \cite[Theorem 2.1(i)]{men1} that $\mathcal{S}^*_{e}\subset \mathcal{S}^*_{1/e}$. Therefore, $\mathcal{S}^*_{1/e}\subset \mathcal{S}^*_e\subset \mathcal{S}^*(1/e)$, which provides the required radii as a consequence of Theorem $\ref{thstar}$.

$(i)$ For the function $f_1(z)=ze^{2z}\sqrt{1+z}$, at $z=-R_{\mathcal{S}_e^*}(\mathcal{T}_1)=-\rho$, we have
\[\left|\log\frac{zf^{\prime}_1(z)}{f_1(z)}\right|=\left|\log\frac{4\rho^2-7\rho+2}{2(1-\rho)}\right|=1.\]
Thus, $R_{\mathcal{S}_e^*}(\mathcal{T}_1)\leq (-2+7e-\sqrt{4+4e+17e^2})/8e$.

 $(ii)$ For the function $f_2(z)=z(1+z)e^z$, at $z=-R_{\mathcal{S}_e^*}(\mathcal{T}_2)=-\rho$, we get
\[\left|\log\frac{zf^{\prime}_2(z)}{f_2(z)}\right|=\left|\log\frac{1-3\rho+\rho^2}{(1-\rho)}\right|=1.\]
Thus, $R_{\mathcal{S}_e^*}(\mathcal{T}_2)\leq(-1+3e-\sqrt{1-2e+5e^2})/2e$.

$(iii)$ For the function $f_3(z)=z(z+\sqrt{1+z^2})^3$, at $z=-R_{\mathcal{S}_e^*}(\mathcal{T}_3)=-\rho$, we obtain
\[\left|\log\frac{zf^{\prime}_3(z)}{f_3(z)}\right|=\left|\log\left(1-\frac{3\rho}{\sqrt{1+\rho^2 }}\right)\right|=1.\]
This proves that $R_{\mathcal{S}_e^*}(\mathcal{T}_3)\leq(1-2e+e^2)/(-1+2e+8e^2)$.

\end{proof}
Sharma \textit{et. al} \cite{jain} studied the various properties of the class $\mathcal{S}^*_c=\mathcal{S}^*(1+(4/3)z+(2/3)z^2$. This class consists the functions $f$ such that $zf^{\prime}(z)/f(z)$ lies in the region bounded by the cardioid
$\Omega_c=\{u+iv:(9u^2+9v^2-18u+5)^2-16(9u^2+9v^2-6u+1)=0\}$. They proved the result that, for $1/3<a\leq 5/3$
\begin{align}\label{card}
\{w\in\mathbb{C}:|w-a|<(3a-1)/3\}\subseteq\Omega_c.
\end{align}
Various results related to this class are investigated in these papers \cite{sharma, sharma3, sharma2, sharma1}. Following corollary provides the radius of cardioid starlikeness for each class $\mathcal{T}_1, \mathcal{T}_2$ and $\mathcal{T}_3$.
\begin{corollary}\label{thcar}
	The following result holds for the class $\mathcal{S}_c^*$.
	\begin{enumerate}
		\item[(i)] $R_{\mathcal{S}_c^*}(\mathcal{T}_1)=1/4= 0.25$
		\item[(ii)] $R_{\mathcal{S}_c^*}(\mathcal{T}_2)=(4-\sqrt{10})/3\approx 0.279241$
		\item[(iii)] $R_{\mathcal{S}_c^*}(\mathcal{T}_3)=2/\sqrt{77}\approx 0.227921$
	\end{enumerate}
\end{corollary}
\begin{proof}
Equation $\eqref{card}$ provides the inclusion $\mathcal{S}^*_{1/3}\subset \mathcal{S}^*_c$ for $a=1$. Thus $R_{\mathcal{S}^*_{1/3}}(\mathcal{T}_i)\leq R_{\mathcal{S}^*_c}(\mathcal{T}_i)$ for $i=1,2,3$. The proof is completed by demonstrating $R_{\mathcal{S}^*_c}(\mathcal{T}_i)\leq R_{\mathcal{S}^*_{1/3}}(\mathcal{T}_i)$ for $i=1,2,3$.

$(i)$ For the function $f_1(z)=ze^{2z}\sqrt{1+z}$, at $z=-R_{\mathcal{S}_c^*}(\mathcal{T}_1)=-\rho$, we obtain
\[\left|\frac{zf^{\prime}_1(z)}{f_1(z)}\right|=\left|\frac{4\rho^2-7\rho+2}{2(1-\rho)}\right|=\frac{1}{3}.\]
Thus, $R_{\mathcal{S}_c^*}(\mathcal{T}_1)\leq1/4$.

$(ii)$ For the function $f_2(z)=z(1+z)e^z$, at $z=-R_{\mathcal{S}_c^*}(\mathcal{T}_2)=-\rho$, we have
\[\left|\frac{zf^{\prime}_2(z)}{f_2(z)}\right|=\left|\frac{1-3\rho+\rho^2}{(1-\rho)}\right|=\frac{1}{3}.\]
Thus, $R_{\mathcal{S}_c^*}(\mathcal{T}_2)\leq(4-\sqrt{10})/3$.

$(iii)$ For the function $f_3(z)=z(z+\sqrt{1+z^2})^3$, at $z=-R_{\mathcal{S}_c^*}(\mathcal{T}_3)=-\rho$,
\[\left|\frac{zf^{\prime}_3(z)}{f_3(z)}\right|=\left|1-\frac{3\rho}{\sqrt{1+\rho^2}}\right|=\frac{1}{3}.\]
This proves that $R_{\mathcal{S}_c^*}(\mathcal{T}_3)\leq2/\sqrt{77}$.
\end{proof}
In 2019, Cho \textit{et al.} \cite{viren} considered the class of starlike functions associated with sine function. The class $\mathcal{S}^*_{sin}$ is defined as $\mathcal{S}^*_{sin}=\{f\in\mathcal{A}:zf^{\prime}(z)/f(z)\prec 1+\sin z:= q_0(z)\}$ for $z\in\mathbb{D}$. For $|a-1|\leq \sin 1$, the following inclusion holds
\begin{align}\label{sin}
\{w\in\mathbb{C}:|w-a|<\sin 1-|a-1|\}\subseteq\Omega_s.
\end{align}
Here $\Omega_s:= q_0(\mathbb{D})$ is the image of the unit disk $\mathbb{D}$ under the mapping $q_0(z)=1+\sin z$.

We find the $\mathcal{S}^*_{sin}$- radius for the classes $\mathcal{T}_1, \mathcal{T}_2$ and $\mathcal{T}_3$.
\begin{corollary}\label{thsin}
	The following results hold for the class $\mathcal{S}_{sin}^*$.
	\begin{enumerate}
		\item[(i)] $R_{\mathcal{S}_{sin}^*}(\mathcal{T}_1)= (5+2\sin1-\sqrt{25-12\sin1+4\sin1^2})/8\approx 0.308961$
		\item[(ii)] $R_{\mathcal{S}_{sin}^*}(\mathcal{T}_2)= (2+\sin1-\sqrt{4+\sin1^2})/2\approx 0.335831$
		\item[(iii)] $R_{\mathcal{S}_{sin}^*}(\mathcal{T}_3)=\sin1/\sqrt{9-\sin1^2}\approx 0.292221$
	\end{enumerate}
\end{corollary}
\begin{proof}
By putting $a=1$ in equation $\eqref{sin}$ we obtain the inclusion $\mathcal{S}^*_{1-sin1}\subset \mathcal{S}^*_{sin}$. Thus $R_{\mathcal{S}^*_{1-sin1}}\leq R_{\mathcal{S}^*_{sin}}$ for $i=1,2,3$. The proof is completed by demonstrating $R_{\mathcal{S}^*_{1-sin1}}\leq R_{\mathcal{S}^*_{sin}}$ for $i=1,2,3$.
$(i)$ For the function $f_1(z)=ze^{2z}\sqrt{1+z}$, at $z=-R_{\mathcal{S}_{sin}^*}(\mathcal{T}_1)=-\rho$, we get
\[\left|\frac{zf^{\prime}_1(z)}{f_1(z)}\right|=\left|\frac{4\rho^2-7\rho+2}{2(1-\rho)}\right|=1+\sin 1.\]
Thus, $R_{\mathcal{S}_{sin}^*}(\mathcal{T}_1)\leq (5+2\sin1-\sqrt{25-12\sin1+4\sin1^2})/8$.

 $(ii)$ For the function $f_2(z)=z(1+z)e^z$, at $z=-R_{\mathcal{S}_{\sin}^*}(\mathcal{T}_2)=-\rho$, we have
\[\left|\frac{zf^{\prime}_2(z)}{f_2(z)}\right|=\left|\frac{1-3\rho+\rho^2}{(1-\rho)}\right|= 1+\sin 1.\]
Thus, $R_{\mathcal{S}_{sin}^*}(\mathcal{T}_2)\leq (2+\sin1-\sqrt{4+\sin1^2})/2$.

 $(iii)$ For the function $f_3(z)=z(z+\sqrt{1+z^2})^3$, at $z=-R_{\mathcal{S}_{\sin}^*}(\mathcal{T}_3)=-\rho$, we get
\[\left|\frac{zf^{\prime}_3(z)}{f_3(z)}\right|=\left|1-\frac{3\rho}{\sqrt{1+\rho^2}}\right|= 1+\sin 1.\]
This proves that $R_{\mathcal{S}_{sin}^*}(\mathcal{T}_3)\leq\sin1/\sqrt{9-\sin1^2}$.
\end{proof}
In the next result, we find the radius for starlike functions accociated with a rational function. Kumar and Ravichandran \cite{sush} introduced the class of starlike functions associated with a rational function, $\psi(z)=1+(z^2k+z^2)/(k^2-kz))$ where $k=\sqrt{2}+1$, defined by $\mathcal{S}_R^*=\mathcal{S}^*(\psi(z))$. For $2(\sqrt{2}-1)<a\leq\sqrt{2}$, they had proved that
\begin{align}\label{rational}
\{w\in\mathbb{C}:|w-a|<a-2(\sqrt{2}-1)\}\subseteq\psi(\mathbb{D}) .
\end{align}
\begin{corollary}\label{thrational}
	The following results holds for the class $\mathcal{S}_R^*$.
	\begin{enumerate}
		\item[(i)] $R_{\mathcal{S}_R^*}(\mathcal{T}_1)=(11-4\sqrt{2}-\sqrt{57-24\sqrt{2}})/8\approx 0.0676475$
		\item[(ii)] $R_{\mathcal{S}_R^*}(\mathcal{T}_2)=(5-2\sqrt{2}-\sqrt{21-12\sqrt{2}})/2\approx 0.0821135$
		\item[(iii)] $R_{\mathcal{S}_R^*}(\mathcal{T}_3)=(\sqrt{-38+27\sqrt{2}})/2\sqrt{14}\approx 0.0572847$
	\end{enumerate}
\end{corollary}
\begin{proof}
	For $a=1$ equation $\eqref{rational}$ gives the inclusion $\mathcal{S}^*_{2(\sqrt{2}-1)}\subset\mathcal{S}^*_R$. Thus $R_{\mathcal{S}^*_{2(\sqrt{2}-1)}}(\mathcal{T}_i)\leq R_{\mathcal{S}^*_R}(\mathcal{T}_i)$ for $i=1,2,3$. We next show that $ R_{\mathcal{S}^*_R}(\mathcal{T}_i)\leq R_{\mathcal{S}^*_{2(\sqrt{2}-1)}}(\mathcal{T}_i)$ for $i=1,2,3$.

	$(i)$ For the function $f_1(z)=ze^{2z}\sqrt{1+z}$, at $z=-R_{\mathcal{S}_{R}^*}(\mathcal{T}_1)=-\rho$, we get
	\[\left|\frac{zf^{\prime}_1(z)}{f_1(z)}\right|=\left|\frac{4\rho^2-7\rho+2}{2(1-\rho)}\right|=2(\sqrt{2}-1).\]
	Thus, $R_{\mathcal{S}_R^*}(\mathcal{T}_1)\leq(11-4\sqrt{2}-\sqrt{57-24\sqrt{2}})/8$.

	$(ii)$ For the function $f_2(z)=z(1+z)e^z$, at $z=-R_{\mathcal{S}_{R}^*}(\mathcal{T}_2)=-\rho$, we have
	\[\left|\frac{zf^{\prime}_2(z)}{f_2(z)}\right|=\left|\frac{1-3\rho+\rho^2}{(1-\rho)}\right|= 2(\sqrt{2}-1).\]
	Thus, $R_{\mathcal{S}_R^*}(\mathcal{T}_2)\leq(5-2\sqrt{2}-\sqrt{21-12\sqrt{2}})/2$.

	$(iii)$ For the function $f_3(z)=z(z+\sqrt{1+z^2})^3$, at $z=-R_{\mathcal{S}_{R}^*}(\mathcal{T}_3)=-\rho$, we obtain
	\[\left|\frac{zf^{\prime}_3(z)}{f_3(z)}\right|=\left|1-\frac{3\rho}{\sqrt{1+\rho^2}}\right|= 2(\sqrt{2}-1).\]
	Thus, $R_{\mathcal{S}_R^*}(\mathcal{T}_3)\leq(\sqrt{-38+27\sqrt{2}})/2\sqrt{14}$.
\end{proof}
In 2020, Wani and Swaminathan \cite{wani} introduced the class $\mathcal{S}^*_{Ne}=\mathcal{S}^*(1+z-(z^3/3))$ that maps open disc $\mathbb{D}$ onto the interior of a two cusped kidney shaped curve $\Omega_{Ne}:= \{u+iv: ((u-1)^2+v^2-4/9)^3-4v^2/3<0\}$. For $1/3<a\leq 1$, they had proved that
\begin{align}\label{ne}
\{w\in\mathbb{C}:|w-a|<a-1/3\}\subseteq \Omega_{Ne}.
\end{align}
Our next theorem determines the $\mathcal{S}^*_{Ne}$ radius results for the classes $\mathcal{T}_1, \mathcal{T}_2$ and $\mathcal{T}_3$.
\begin{corollary}\label{thne}
	The following results hold for the class $\mathcal{S}^*_{Ne}$.
	\begin{enumerate}
		\item[(i)] $R_{\mathcal{S}^*_{Ne}}(\mathcal{T}_1)= 1/4=0.25$
		\item [(ii)] $R_{\mathcal{S}^*_{Ne}}(\mathcal{T}_2)=(4-\sqrt{10})/3\approx 0.279241$
		\item [(iii)] $R_{\mathcal{S}^*_{Ne}}(\mathcal{T}_3)=2/\sqrt{77}\approx 0.227921$
	\end{enumerate}
\end{corollary}
\begin{proof}
From equation $\eqref{ne}$ we obtain the inclusion $\mathcal{S}^*_{1/3}\subset \mathcal{S}^*_{Ne}$ for $a=1$. This shows that $R_{\mathcal{S}^*_{1/3}}(\mathcal{T}_i)\leq R_{\mathcal{S}^*_{Ne}}(\mathcal{T}_i)$ for $i=1,2,3$.

 $(i)$ For the function $f_1(z)=ze^{2z}\sqrt{1+z}$, at $z=-R_{\mathcal{S}_{Ne}^*}(\mathcal{T}_1)=-\rho$,
\[\left|\frac{zf^{\prime}_1(z)}{f_1(z)}\right|=\left|\frac{4\rho^2-7\rho+2}{2(1-\rho)}\right|=\frac{1}{3}.\]
Thus, $R_{\mathcal{S}^*_{Ne}}(\mathcal{T}_1)\leq 1/4$.

$(ii)$ For the function $f_2(z)=z(1+z)e^z$, at $z=-R_{\mathcal{S}_{R}^*}(\mathcal{T}_2)=-\rho$,
\[\left|\frac{zf^{\prime}_2(z)}{f_2(z)}\right|=\left|\frac{1-3\rho+\rho^2}{(1-\rho)}\right|=\frac{1}{3}.\]
Thus, $R_{\mathcal{S}^*_{Ne}}(\mathcal{T}_2)\leq(4-\sqrt{10})/3$.

$(iii)$ For the function $f_3(z)=z(z+\sqrt{1+z^2})^3$, at $z=-R_{\mathcal{S}_{Ne}^*}(\mathcal{T}_3)=-\rho$,
\[\left|\frac{zf^{\prime}_3(z)}{f_3(z)}\right|=\left|1-\frac{3\rho}{\sqrt{1+\rho^2}}\right|= \frac{1}{3}.\]
This proves that $R_{\mathcal{S}^*_{Ne}}(\mathcal{T}_3)\leq2/\sqrt{77}$.
\end{proof}
Goel and Kumar \cite{goel} introduced the class $\mathcal{S}^*_{SG}:=\mathcal{S}^*(2/1+e^{-z})$, where $2/(1+e^{-z})$ is a modified sigmoid function that maps $\mathbb{D}$ onto the region $\Omega_{SG}:=\{w=u+iv:|\log(w/(2-w))|<1\}$. Precisely, $f\in\mathcal{S}^*_{SG}$ provided the function $zf^{\prime}(z)/f(z)$ maps $\mathbb{D}$ onto the region lying inside the domain $\Omega_{SG}$. For $2/(e+1)<a<2e/(1+e)$, Goel and Kumar \cite{goel} had proved the following inclusion
\begin{align}\label{sigmoid}
\{w\in\mathbb{C}:|w-a|<r_{SG}\}\subset \Omega_{SG},
\end{align}
provided $r_{SG}=((e-1)/(e+1))-|a-1|$.
 Next result is about the $\mathcal{S}^*_{SG}$ radius for the defined  classes.
\begin{theorem}\label{thsigm}
 The following results hold for the class $\mathcal{S}^*_{SG}$.
 \begin{enumerate}
 	\item[(i)] $R_{\mathcal{S}^*_{SG}}(\mathcal{T}_1)=(3+7e-\sqrt{41+42e+17e^2})/(8+8e)\approx 0.177213$
 	\item [(ii)]$R_{\mathcal{S}^*_{SG}}(\mathcal{T}_2)=(1+3e-\sqrt{5+6e+5e^2})(2+2e)\approx 0.204712$
 	\item[(iii)] $R_{\mathcal{S}^*_{SG}}(\mathcal{T}_3)=(\sqrt{1-2e+e^2})/(8+20e+8e^2)\approx 0.1559$
 \end{enumerate}
\end{theorem}
\begin{proof}
$(i)$ The function defined by $m(r)= (4r^2-7r+2)/2(1-r), 0\leq r<1$ is a decreasing function. Let $\rho=R_{\mathcal{S}^*_{SG}}(\mathcal{T}_1)$ is the root of the equation $m(r)=2/(1+e)$. For $0<r\leq R_{\mathcal{S}^*_{SG}}(\mathcal{T}_1)$, we have $m(r)\geq 2/(1+e)$. That is
\[\frac{r(5-4r)}{2(1-r)}\leq \frac{e-1}{e+1}.\]
For the class $\mathcal{T}_1$, the centre of the disk is $1$, therefore the disk obtained in $\eqref{disk1}$ is contained in the region bounded by modified sigmoid, by equation $\eqref{sigmoid}$. For the function $f_1(z)=ze^{2z}\sqrt{1+z}$, at $z=-R_{\mathcal{S}^*_{SG}}(\mathcal{T}_1)=-\rho$, we have
\[\left|\log\frac{zf^{\prime}_1(z)/f_1(z)}{2-(zf^{\prime}_1(z)/f_1(z))}\right|=\left|\frac{(4\rho^2-7\rho+2)/2(1-\rho)}{2-((4\rho^2-7\rho+2)/2(1-\rho))}\right|=1.\]

$(ii)$ The function defined $n(r)=(r^2-3r+1)/(1-r), 0\leq r<1$ is a decreasing function. Let $\rho=R_{\mathcal{S}_{SG}^*}(\mathcal{T}_2)$ is the root of the equation $n(r)=1/3$. For $0<r\leq R_{\mathcal{S}_{SG}^*}(\mathcal{T}_2)$, we have $n(r)\geq 2/1+e$. That is,
\[\frac{r(-2+r)}{(-1+r)}\leq \frac{e-1}{e+1}.\]
For the class $\mathcal{T}_2$, the centre of the disk is $1$, therefore the disk obtained in $\eqref{disk2}$ is contained in the region bounded by the modified sigmoid, using equation $\eqref{sigmoid}$. For the function $f_2(z)=z(1+z)e^z$, at $z=-R_{\mathcal{S}^*_{SG}}(\mathcal{T}_3)=-\rho$,
\[\left|\log\frac{zf^{\prime}_2(z)/f_2(z)}{2-(zf^{\prime}_2(z)/f_2(z))}\right|=\left|\frac{(1-3\rho+\rho^2)/(1-\rho)}{2-((1-3\rho+\rho^2)/(1-\rho))}\right|=1.\]

$(iii)$  The function defined $s(r)= 1-3r/\sqrt{1+r^2}, 0\leq r<1$ is a decreasing function. Let
$\rho=R_{\mathcal{S}^*_{SG}}(\mathcal{T}_3)$ is the root of the equation $s(r)=2/(1+e)$. For $0<r\leq R_{\mathcal{S}^*_{SG}}(\mathcal{T}_3)$, we have $s(r)\geq 2/1+e$. That is
\[\frac{3r}{\sqrt{1+r^2}}\leq\frac{e-1}{e+1}.\]
For the class $\mathcal{T}_3$, the centre of the disk is 1, therefore the disk obtained in $\eqref{disk}$ is contained in the region bounded by the modified sigmoid, by equation $\eqref{sigmoid}$. For the function $f_2(z)=z(z+\sqrt{1+z^2})^3$, at $z=-R_{\mathcal{S}^*_{SG}}(\mathcal{T}_3)=-\rho$,
\[\left|\log\frac{zf^{\prime}_3(z)/f_3(z)}{2-(zf^{\prime}_3(z)/f_3(z))}\right|=\left|\frac{1-(3\rho)/\sqrt{1+\rho^2}}{2-(1-(3\rho)/(\sqrt{1+\rho^2}))}\right|=1.\]
\end{proof}

\end{document}